\documentclass[12pt,a4paper]{amsart}
\usepackage{mathrsfs}
\usepackage{amssymb,amsmath,amsthm,color}
\usepackage{graphicx,mcite}
\usepackage{cite}
\usepackage{hyperref}
\usepackage{url}
\usepackage{setspace}
\usepackage{enumerate}
\newtheorem{theorem}{Theorem}[section]
\newtheorem{lemma}[theorem]{Lemma}
\newtheorem{proof of lemma}[theorem]{Proof of Lemma}
\newtheorem{proposition}[theorem]{Proposition}

\theoremstyle{definition}

\newtheorem{remark}[theorem]{Remark}

\numberwithin{equation}{section}

\begin{document}

\title[uniqueness of the Fourier transform]
{Uniqueness of the Fourier transform on certain Lie groups}

\author{A. Chattopadhyay, D.K. Giri and R.K. Srivastava}

\address{Department of Mathematics, Indian Institute of Technology, Guwahati, India 781039.}
\email{{arupchatt@iitg.ac.in, deb.giri@iitg.ac.in, rksri@iitg.ac.in}}


\date{\today}



\begin{abstract}
In this article, we prove that if the group Fourier transform of certain integrable
functions on the Heisenberg motion group (or step two nilpotent Lie groups) is of
finite rank, then the function is identically zero. These results can be thought
as an analogue to the Benedicks theorem that dealt with the uniqueness of the Fourier
transform of integrable functions on the Euclidean spaces.
\end{abstract}

\maketitle

\section{Introduction}\label{section1}
In an interesting article, M. Benedicks \cite{B} had extended the classical Paley-Wiener
theorem for compactly supported function to the class of integrable functions. In other
words, support of an integrable function  and its Fourier transform both cannot be of
finite measure simultaneously. Thereafter, a series of analogous results to the
Benedicks theorem has been explored in various contexts, including the Heisenberg group
and the Euclidean motion groups (see \cite{NR,PS,PSi,PT,R,SST}). In article \cite{NR},
an analogous result on the Heisenberg group has worked out for the partial compactly
supported functions in terms of finite rank of Fourier transform of the function.
Further, Vemuri \cite{V} has relaxed the compact support condition on the functions
by finite Lebesgue measure.

\smallskip

In this article, we explore analogous results to the Amrein-Berthier and
Benedicks theorem on the Heisenberg motion group and step two nilpotent Lie groups. We prove
that if the group Fourier transform of finitely supported certain integrable functions on the
Heisenberg motion group (or step two nilpotent Lie groups) is of finite rank, then the function
has to vanish identically. However, it would be a reasonable to consider the case when the
spectrum of the Fourier transform of an integrable function will be supported on a thin
uncountable set.

\section{Preliminaries on the Heisenberg motion group}\label{section2}
The Heisenberg group $\mathbb H^n=\mathbb C^n\times\mathbb R$ is a step
two nilpotent Lie group having center $\mathbb R$ that equipped with the group law
\[(z,t)\cdot(w,s)=\left(z+w,t+s+\frac{1}{2}\text{Im}(z\cdot\bar w)\right).\]
By Stone-von Neumann theorem, the infinite dimensional irreducible unitary
representations of $\mathbb H^n$ can be parameterized by $\mathbb R^\ast=\mathbb R\smallsetminus\{0\}.$
That is, each of $\lambda\in\mathbb R^\ast$ defines a Schr\"{o}dinger representation $\pi_\lambda$ of
$\mathbb H^n$ by
\[\pi_\lambda(z,t)\varphi(\xi)=e^{i\lambda t}e^{i\lambda(x\cdot\xi+\frac{1}{2}x\cdot y)}\varphi(\xi+y),\]
where $z=x+iy$ and $\varphi\in L^2(\mathbb{R}^n).$
Let \[T=\frac{\partial}{\partial t},~ X_j=\frac{\partial}{\partial
x_j}+\frac{1}{2}y_j\frac{\partial}{\partial t}~\text{and}~
Y_j=\frac{\partial}{\partial y_j}-\frac{1}{2}x_j\frac{\partial}{\partial t}.\]
Then $\{T,X_j,Y_j: j=1,\ldots,n\}$ forms a  basis for the  Lie algebra $\mathfrak h^n$
consists of all left-invariant vector fields on $\mathbb H^n$ and the representation
$\pi_\lambda$ induces a representation $\pi_\lambda^*$ of $\mathfrak h^n$ on the space
of $C^\infty$ vectors in $L^2(\mathbb R^n)$ via
\[\pi_\lambda^*(X)f=\left.\frac{d}{dt}\right\vert_{t=0}\pi_\lambda(\exp tX)f.\]

\bigskip

It is easy to see that $\pi_\lambda^*(X_j)=i\lambda x_j$ and $\pi_\lambda^*(Y_j)=\frac{\partial}{\partial x_j}.$
Hence for the sub-Laplacian $\mathcal L=-\sum_{j=1}^n(X_j^2+Y_j^2),$
it follows that $\pi_\lambda^*(\mathcal L)=-\Delta_x+\lambda^2|x|^2=:H_\lambda,$
the scaled Hermite operator. Let $\phi_\alpha^\lambda(x)=|\lambda|^{\frac{n}{4}}\phi_\alpha(\sqrt{|\lambda|}x);~\alpha\in\mathbb Z_+^n,$
where $\phi_\alpha$ are the Hermite functions on $\mathbb R^n.$
Then $ \phi_\alpha^\lambda$'s are the eigenfunctions of $H_\lambda$ with eigenvalue
$(2|\alpha|+n)|\lambda|.$ Hence the entry functions $E_{\alpha\beta}^\lambda$'s
of the representation $\pi_\lambda$ are eigenfunctions of the sub-Laplacian $\mathcal L$
satisfying \[\mathcal L E_{\alpha\beta}^\lambda=(2|\alpha|+n)|\lambda|E_{\alpha\beta}^\lambda,\]
where $E_{\alpha\beta}^\lambda(z,t)=\left\langle\pi_\lambda(z,t)\phi_\alpha^\lambda,\phi_\beta^\lambda\right\rangle.$
Since $E_{\alpha\beta}^\lambda(z,t)=e^{i\lambda t}\left\langle\pi_\lambda(z)\phi_\alpha^\lambda,\phi_\beta^\lambda\right\rangle,$
the eigenfunctions $E_{\alpha\beta}^\lambda$'s are not in $L^2(\mathbb H^n).$
However, for a fix $t,$ they are in $L^2(\mathbb C^n).$ Now, define an operator
$L_\lambda$ by $\mathcal L\left(e^{i\lambda t}f(z)\right)=e^{i\lambda t}L_\lambda f(z).$
Then the special Hermite functions
\[\phi_{\alpha\beta}^\lambda(z)=(2\pi)^{-\frac{n}{2}}\left\langle\pi_\lambda(z)\phi_\alpha^\lambda,\phi_\beta^\lambda\right\rangle\]
are eigenfunctions of $L_\lambda$ with eigenvalue $2|\alpha|+n.$ We summarize by
noting that the special Hermite functions $\phi^\lambda_{\alpha\beta}$'s forms
an orthonormal basis for $L^2(\mathbb C^n)$ (see \cite{T2}, Theorem 2.3.1).

\bigskip

Heisenberg motion group $G$ is the group of isometries of $\mathbb H^n$ that
leaves invariant the sub-Laplacian $\mathcal L.$ Since the action of the unitary
group $K=U(n)$ defines a group of automorphism on $\mathbb H^n$ via $k\cdot(z,t)=(kz,t),$
where $k\in K,$ the group $G$ can be  expressed as the semidirect product of
$\mathbb H^n$ and $K.$ Hence the group law on $G$ can be understood by
\[(k_1, z, t)\cdot(k_2, w, s)=\left(k_1k_2, z+k_1w, t+s-\frac{1}{2}\text{Im}(k_1w\cdot\bar z)\right).\]
Since a right $K$-invariant function on $G$ can be thought as a function
on $\mathbb H^n,$  we infer that the Haar measure on $G$ can be written as
$dg=dkdzdt,$ where $dk$ and $dzdt$ are the normalized Haar measure on $K$
and $\mathbb H^n$ respectively.

\smallskip

For $k\in K,$ define another set of representations of the Heisenberg group
$\mathbb H^n$ by $\pi_{\lambda,k}(z,t)=\pi_\lambda(kz,t).$ Since $\pi_{\lambda,k}$
agrees with $\pi_\lambda$ on the center of $\mathbb H^n,$  it follows by
the Stone-Von Neumann theorem for the Schr\"{o}dinger representation that
$\pi_{\lambda,k}$ is equivalent to $\pi_\lambda.$ Hence there exists an
intertwining operator $\mu_\lambda(k)$ satisfying
\begin{equation}\label{exp05}
\pi_\lambda(kz,t)=\mu_\lambda(k)\pi_\lambda(z,t)\mu_\lambda(k)^\ast.
\end{equation}
The operator-valued function $\mu_\lambda$ can be thought as a unitary
representation of the group $K$ on $L^2(\mathbb R^n)$ and it is known as
metaplectic representation. Since for $\lambda\in\mathbb R^\ast,$ the set
$\{\phi_\alpha^\lambda : \alpha\in\mathbb N^n \}$ forms an orthonormal basis
for $L^2(\mathbb R^n),$ let $P_m=\{\phi_\alpha^\lambda:~ |\alpha|=m\}.$
Then $\mu_\lambda\vert_{P_m}$ is an irreducible representation of $K$ and
the action of $\mu_\lambda$ on $L^2(\mathbb R^n)$ can be realized by
\begin{equation}\label{exp09}
\mu_\lambda(k)\phi_\alpha^\lambda=\sum_{|\gamma|=|\alpha|}\eta_{\alpha\gamma}^\lambda(k)\phi_\gamma^\lambda.
\end{equation}
For more details about the metaplectic representations and the spherical
functions on $\mathbb H^n,$ we refer the article by Benson et al. \cite{BJR}.
Let $(\sigma,\mathcal H_\sigma)$ be an irreducible unitary representation of $K$
and $\mathcal H_\sigma=\text{span}\{e_j^\sigma:1\leq j\leq d_\sigma\}.$
For $k\in K,$ the matrix coefficients of the representation $\sigma\in\hat K,$
are define \[\varphi_{ij}^\sigma(k)=\langle\sigma(k)e_j^\sigma, e_i^\sigma\rangle.\]

Define a bilinear form $\phi_{\alpha}^\lambda\otimes e_i^\sigma$ on $L^2(\mathbb R^n)\times\mathcal H_\sigma$
by $\phi_{\alpha}^\lambda\otimes e_i^\sigma=\phi_{\alpha}^\lambda~e_i^\sigma.$
Then the set $\{\phi_{\alpha}^\lambda\otimes e_i^\sigma:1\leq i\leq d_\sigma, \alpha\in\mathbb N^n\}$
forms an orthonormal basis for $L^2(\mathbb R^n)\otimes\mathcal{H}_\sigma.$
Denote  $\mathcal H_\sigma^2=L^2(\mathbb R^n)\otimes\mathcal{H}_\sigma.$
\smallskip

For $\lambda\neq0,$ we define a representation $\rho_\sigma^\lambda$ of $G$
on the space $\mathcal H_\sigma^2$ by
\[\rho_\sigma^\lambda(z,t,k)=\pi_\lambda(z,t)\mu_\lambda(k)\otimes\sigma(k).\]
In the article \cite{S}, it has been shown that $\rho_\sigma^\lambda$ are the only
irreducible unitary representations of $G$ which appears in the Plancherel formula.
Thus, in view of the above argument, we denote the partial dual of the group $G$ by
$ G'\cong\mathbb R^\ast\times\hat K.$

\smallskip

Now, we define the Fourier transform of the function $f\in L^1(G)$ by
\[\hat f(\lambda,\sigma) =\int_K\int_\mathbb{R}\int_{\mathbb{C}^n} f(z,t,k)\rho_\sigma^\lambda(z,t,k)dzdtdk.\]
Let $f^\lambda$ be the inverse Fourier transform of the function $f$ in $t$ variable.
Then \[f^\lambda(z,k)=\int_\mathbb{R}f(z,t,k)e^{i\lambda t}dt.\]
Thus,
\[\hat f(\lambda,\sigma)=\int_K\int_{\mathbb{C}^n} f^\lambda(z,k)\rho_\sigma^\lambda(z,k)dzdk,\]
where $\rho_\sigma^\lambda(z,k)=\rho_\sigma^\lambda(z,0,k).$
For $f\in L^1\cap L^2(G),$ the following Plancherel formula derived in \cite{S}.
\[\int_K\int_{\mathbb{H}^n} f(z,t,k)dzdtdk=(2\pi)^{-n}\sum_{\sigma\in\hat K}
\int_{\mathbb{R}\setminus\{0\}}\|\hat f(\lambda,\sigma)\|^2_{HS}|\lambda|^2d\lambda.\]
Further, the set $\{\phi_{\alpha}^\lambda\otimes e_i^\sigma : \alpha\in\mathbb N^n,~1\leq i\leq d_\sigma\}$
forms an orthonormal basis for $\mathcal{H}_\sigma^2,$ we can write
\[\hat f(\lambda,\sigma)(\phi_{\gamma}^\lambda\otimes e_i^\sigma)=\sum_{|\alpha|=|\gamma|}\int_K\eta_{\alpha\gamma}^\lambda(k)
\int_{\mathbb{C}^n} f^\lambda(z,k)\left(\pi_\lambda(z)\phi_{\alpha}^\lambda\otimes\sigma(k)e_i^\sigma\right)dzdk.\]

\section{Uniqueness results on the Heisenberg motion group}\label{section3}
In this section, we work out some of the results pertaining to the uniqueness
of the Fourier transform on the Heisenberg motion group $G=\mathbb{H}^n\ltimes K.$
Those results can be thought as an analogue to the Benedicks theorem.

\bigskip

\noindent \textbf{Weyl transform.}
For proving the main result of this section, we need to derive some of the properties
of the Weyl type transform on $G^\times=\mathbb C^n\times K.$ For more details
on the Wely transform on the Heisenberg group, see \cite{T}.

\smallskip

For $(\lambda, \sigma)\in G',$ we define the Weyl transform $W_\sigma^\lambda$ on
$L^1(G^\times)$ by
\[W_\sigma^\lambda(F)=\int_{K}\int_{\mathbb{C}^n}F(z,k)\rho_\sigma^\lambda(z,k)dzdk.\]
Now, we define the $\lambda$-twisted convolutions of $F,H\in L^1\cap L^2(G^\times)$ by
\[F\times_\lambda H(g)=\int_{G^\times}F\left(g{g'}^{-1}\right)H(g')e^{-\frac{i}{2}\lambda\text{Im}(kw.\bar z)}dg',\]
where $g=(z,k)$ and $g'=(w,s).$ For $\lambda=1,$ we simply call the $\lambda$-twisted
convolutions as twisted convolutions and denote it by $F\times H.$ We derive the
following properties of the Weyl transform $W_\sigma^\lambda.$

\begin{proposition}\label{prop2}
If $F, H\in L^1\cap L^2(G^\times),$ then\\
$(i)~W_\sigma^\lambda(F^\ast)=W_\sigma^\lambda(F)^\ast,$ where $F^\ast(z,k)=\overline{F\left((z,k)^{-1}\right)},$\\
$(ii)~W_\sigma^\lambda(F\times_\lambda H)=W_\sigma^\lambda(F)W_\sigma^\lambda(H).$
\end{proposition}
\begin{proof}
By the scaling argument, it is enough to prove these results for the case $\lambda=1.$
\smallskip

$(i)$ If $\phi,\psi\in\mathcal{H}_\sigma^2,$ then we have
\begin{eqnarray*}
\left\langle W_\sigma(F^\ast)\phi,\psi\right\rangle&=&\int_{K}\int_{\mathbb C^n}F^\ast(z,k)\left\langle\rho_\sigma(z,k)\phi,\psi\right\rangle dzdk\\
&=&\int_K\int_{\mathbb C^n}\left\langle\phi,F\left((z,k)^{-1}\right)\rho_\sigma\left((z,k)^{-1}\right)\psi\right\rangle dzdk\\
&=&\left\langle\phi,W_\sigma(F)\psi\right\rangle=\left\langle W_\sigma(F)^\ast\phi,\psi\right\rangle.
\end{eqnarray*}

$(ii)$ Let $dg=dzdk,$ then
\begin{eqnarray*}
\left\langle W_\sigma(F)W_\sigma(H)\phi,\psi\right\rangle&=&\int_{G^\times}F(z,k)\left\langle\rho_\sigma(z,k)W_\sigma(H)\phi,\psi\right\rangle dg\\
&=&\int_{G^\times}\int_{G^\times}F(g)H(g')e^{-\frac{i}{2}\text{Im}(kw.\bar z)}\left\langle\rho_\sigma(z+kw,ks)\phi,\psi\right\rangle dg'dg\\
&=&\int_{G^\times}\int_{G^\times}F\left(g{g'}^{-1}\right)H(g')e^{-\frac{i}{2}\text{Im}(kw.\bar z)}\left\langle\rho_\sigma(g)\phi,\psi\right\rangle dg'dg\\
&=&\int_{G^\times}(F\times H)(z,k)\left\langle\rho_\sigma(z,k)\phi, \psi\right\rangle dg\\
&=&\left\langle W_\sigma(F\times H)\phi, \psi\right\rangle.
\end{eqnarray*}
\end{proof}
Next, we derive the Plancherel formula for the Weyl transform $W_\sigma^\lambda$
on $L^2(G^\times)$ corresponding to $\lambda=1.$

\begin{proposition}\label{prop05}
If $F\in L^2(G^\times),$ then the following holds.
\[\sum_{\sigma\in\hat K}d_\sigma\left|\left|W_\sigma(F)\right|\right|_{HS}^2=(2\pi)^n\int_{K}\int_{\mathbb{C}^n}|F(z,k)|^2dzdk.\]
\end{proposition}
\begin{proof}
Since $L^1\cap L^2(G^\times)$ is dense in $L^2(G^\times),$ it is enough to prove the
result for $L^1\cap L^2(G^\times).$ For the sake of convenience, let
$\phi_{\alpha, i}^{\sigma}=\phi_{\alpha}^\lambda\otimes e_i^\sigma$
and $\phi_{\alpha\beta}=(2\pi)^{\frac{n}{2}}\phi_{\alpha\beta}^\lambda$ when $\lambda=1.$
Then the set $\{\phi_{\gamma,i}^\sigma :\gamma\in\mathbb N^n, 1\leq i\leq d_\sigma\}$
forms an orthonormal basis for $\mathcal{H}_\sigma^2.$  By the Parseval identity,
we have
\[\left\|W_\sigma(F)\phi_{\gamma,i}^\sigma\right\|_{\mathcal{H}_\sigma^2}^2=\sum_{\beta\in\mathbb N^n}\sum_{j=1}^{d_\sigma}\left|\left\langle W_\sigma(F)\phi_{\gamma,i}^\sigma,\phi_{\beta,j}^\sigma\right\rangle\right|^2=\]
 \[(2\pi)^n\sum_{\beta\in\mathbb N^n}\sum_{j=1}^{d_\sigma}\left|\sum_{|\alpha|=|\gamma|}\int_{K}\eta_{\alpha\gamma}(k)\int_{\mathbb{C}^n}F(z,k)\phi_{\alpha\beta}(z)\varphi_{ji}^\sigma(k)dzdk\right|^2.
\]
It is easy to see that the matrix coefficients $\eta_{\alpha\gamma}$ of the representation
$\mu_\lambda$ satisfy the identity
\begin{eqnarray}\label{exp33}
\sum_{|\alpha|=m}\left|\sum_{|\gamma|=m}c_\alpha\eta_{\alpha\gamma}(k)\right|^2=\sum_{|\alpha|=m}\left|c_\alpha\eta_{\alpha\gamma}(k)\right|^2,
\end{eqnarray}
where $k\in K$ and $c_\alpha\in\mathbb C.$ Now, by Plancherel theorem for the
compact group $K$ and the identity (\ref{exp33}), we infer that
\begin{eqnarray*}
\sum_{\sigma\in\hat K}d_\sigma\left|\left|W_\sigma(F)\right|\right|_{HS}^2 &=& (2\pi)^n\sum_{\beta, \gamma\in\mathbb N^n}\int_{K}\left|\sum_{|\alpha|=|\gamma|}\eta_{\alpha\gamma}(k)\int_{\mathbb{C}^n}F(z,k)\phi_{\alpha\beta}(z)dz\right|^2dk\\
&=& (2\pi)^n\sum_{\alpha,\beta\in\mathbb N^n}\int_{K}\left|\int_{\mathbb{C}^n}F(z,k)\phi_{\alpha\beta}(z)dz\right|^2dk\\
&=& (2\pi)^n\int_{K}\int_{\mathbb{C}^n}|F(z,k)|^2dzdk.
\end{eqnarray*}
\end{proof}

For $\sigma\in\hat K,$ we defining a Fourier-Wigner type transform $V_f^g$ of
functions $f,g\in \mathcal{H}_\sigma^2$ on $G^\times$ by
\[V_f^g(z,k)=\left\langle\rho_\sigma(z,k)f,g\right\rangle.\]

\begin{lemma}\label{lemma1}
For $f_l, g_l\in \mathcal{H}_\sigma^2,~l=1,2,$ the following identity holds.
\[\int_{K}\int_{\mathbb{C}^n}V_{f_1}^{g_1}(z,k)\overline{V_{f_2}^{g_2}(z,k)}dzdk=
(2\pi)^n\left\langle f_1,f_2 \right\rangle\overline{\left\langle g_1,g_2 \right\rangle}.\]
\end{lemma}

\begin{proof}
Since the set $\{\phi_\alpha\otimes e_i^\sigma :\alpha\in\mathbb N^n, 1\leq i\leq d_\sigma\}$
form an orthonormal basis for $\mathcal{H}_\sigma^2$ and $f_l, g_l\in \mathcal{H}_\sigma^2,$
we can write
\[f_l=\sum_{\gamma\in\mathbb{N}^n}\sum_{1\leq i\leq d_\sigma}f_{\gamma,i}^l\phi_\gamma\otimes e_i^\sigma,
\text{ and } g_l=\sum_{\beta\in\mathbb{N}^n}\sum_{1\leq j\leq d_\sigma}g_{\beta,j}^l\phi_\beta\otimes e_j^\sigma, ~l=1,2,\]
where $f_{\gamma,i}^l$ and $g_{\beta,j}^l$ are constants.
Thus,
\[V_{f_l}^{g_l}(z,k)=(2\pi)^{\frac{n}{2}}\sum_{\alpha,\beta\in\mathbb N^n}\sum_{1\leq i,j\leq d_\sigma}\sum_{|\gamma|=|\alpha|}
f_{\gamma,i}^l\overline{g_{\beta,j}^l}\eta_{\alpha\gamma}(k)
\phi_{\alpha\beta}(z)\varphi_{ji}^\sigma(k),\]
By the
orthogonality of the special Hermite functions $\phi_{\alpha\beta}$ together with the
identity (\ref{exp33}), it follows that
\[\int_{\mathbb{C}^n}V_{f_1}^{g_1}(z,k)\overline{V_{f_2}^{g_2}(z,k)}dz=\]
\[(2\pi)^n\sum_{\gamma,\beta\in\mathbb N^n}\left[\sum_{i,j=1}^{d_\sigma}\left(f_{\gamma,i}^1\overline{g_{\beta,j}^1}\right)\phi_{ji}^\sigma(k)
\sum_{i,j=1}^{d_\sigma}\left(\overline{f_{\gamma,i}^2}g_{\beta,j}^2\right)\overline{\phi_{ji}^\sigma(k)}\right].\]
Finally, by integrating both the sides with respect to $k,$ we get

\[\int_{K}\int_{\mathbb{C}^n}V_{f_1}^{g_1}(z,k)\overline{V_{f_2}^{g_2}(z,k)}dzdk=\]
\[(2\pi)^n\left(\sum_{\gamma\in\mathbb N^n}\sum_{1\leq i\leq d_\sigma}f_{\gamma,i}^1\overline{f_{\gamma,i}^2}\right)\left(\sum_{\beta\in\mathbb N^n}
\sum_{1\leq j\leq d_\sigma}g_{\beta,j}^2\overline{g_{\beta,j}^1}\right)=
(2\pi)^n\left\langle f_1,f_2 \right\rangle\overline{\left\langle g_1,g_2 \right\rangle}.\]

\end{proof}

Notice that, if $f,g\in \mathcal{H}_\sigma^2,$ then as particular case of Lemma \ref{lemma1},
it follows that $V_f^g\in L^2(G^\times).$ Let $V_\sigma=\overline{\text{span}}\left\{V_f^g: f,g\in \mathcal{H}_\sigma^2\right\}.$
Since the set \[B_\sigma=\left\{\psi_{\alpha,i}^\sigma=\phi_\alpha\otimes e_i^\sigma :\alpha\in\mathbb N^n, 1\leq i \leq d_\sigma\right\}\]
form an orthonormal basis for $\mathcal{H}_\sigma^2,$ by Lemma \ref{lemma1},
we infer that the set
\[V_{B_\sigma}=\left\{V_{\psi_{\alpha,i}^\sigma}^{\psi_{\beta,j}^\sigma}:~
\psi_{\alpha,i}^\sigma,\psi_{\beta,j}^\sigma\in B_\sigma\right\}\]
is an orthonormal basis for $V_\sigma.$ Next, we recall the
Peter-Weyl theorem which is crucial for the proof of Proposition \ref{prop5}.
For more details, see \cite{Su}.

\begin{theorem}\label{th10} \em{(Peter-Weyl).} Let $\hat K$ be the unitary
dual of the compact Lie group $K.$ Then the set $\left\{\sqrt{d_\sigma}\phi_{ij}^\sigma:1\leq i,j\leq d_\sigma, \sigma\in\hat K\right\}$
is an orthonormal basis for the space $L^2(K).$
\end{theorem}

\begin{proposition}\label{prop5}
The set $\left\{V_{B_\sigma}: \sigma\in\hat K\right\}$ is an orthonormal basis for $L^2(G^\times).$
\end{proposition}

\begin{proof}
By Theorem \ref{th10}, it follows that $\left\{V_{B_\sigma}: \sigma\in\hat K\right\}$
is an orthonormal set. It only remains to prove the completeness. For this,
suppose $F\in V_{B_\sigma}^\perp,$ then
\begin{eqnarray*}
\left\langle W_\sigma(\overline{F})\psi_{\alpha,i}^\sigma,\psi_{\beta,j}^\sigma\right\rangle
&=& \int_{K}\int_{\mathbb{C}^n}\overline{F}(z,k)V_{\psi_{\alpha,i}^\sigma}^{{\psi_{\beta,j}^\sigma}}(z,k)dzdk\\
&=& \left\langle F,V_{\psi_{\alpha,i}^\sigma}^{\psi_{\beta,j}^\sigma} \right\rangle=0,\nonumber
\end{eqnarray*}
whenever $\psi_{\alpha,i}^\sigma,\psi_{\beta,j}^\sigma\in B_\sigma.$ Hence, it follows that
$W_\sigma(\overline{F})=0$ for all $\sigma\in\hat K.$ Thus, by Proposition \ref{prop05},
we conclude that $F=0.$
\end{proof}
Moreover, by using the fact that $V_{B_\sigma}$ is an orthonormal basis for $V_{\sigma},$
as a corollary to Proposition \ref{prop5}, we infer that $L^2(G^\times)=\bigoplus\limits_{\sigma\in\hat K}V_{\sigma}.$

\smallskip

Now, we state our main result of this section. Let $\text{A}$ and $\text{B}$ are
Lebesgue measurable subsets of $\mathbb{R}^n$ such that $0<m(\text{A}),m(\text{B})$
$<\infty,$ where $m$ denotes the Lebesgue  measure on $\mathbb R^n.$

\begin{theorem}\label{th2}
Let $F\in L^1\cap L^2(G)$ be supported on $(\Sigma\times\mathbb{R})\times K.$

\smallskip

$(i)$ If $~\Sigma$ has finite Lebesgue measure and $\hat{F}(\lambda,\sigma)$ is a rank
one operator for all $(\lambda,\sigma)\in\mathbb R^\ast\times\hat K,$ then $F=0.$
\smallskip

$(ii)$ If $\Sigma=\text{A}\times\text{B}$ and $\hat{F}(\lambda,\sigma)$ is a finite
rank operator for all $(\lambda,\sigma)\in\mathbb R^\ast\times\hat K,$ then $F=0.$
\end{theorem}

Given $\phi,\psi\in L^2(\mathbb{R}^n),$ we define the Fourier-Wigner transform by
\[T(\phi,\psi)(z)=\left\langle\pi(z)\phi,\psi\right\rangle,\]
where $\pi$ is the Schr\"{o}dinger representation corresponding to $\lambda=1.$ Next,
we state the following result from \cite{Jam, Jan}.
\begin{theorem}\label{thm001}
For $\phi,\psi\in L^2(\mathbb{R}^n),$ write $X=T(\phi,\psi).$ If
$\left\{z\in\mathbb C^n:~X(z)\neq0\right\}$ has finite Lebesgue measure, then $X=0.$
\end{theorem}
In view of Theorem \ref{thm001}, we prove the following analogous result for the
Fourier-Wigner transform. In fact, it says that the Fourier-Wigner transform of
a pair of non-zero functions cannot be finitely supported.
\begin{proposition}\label{prop1}
For $f_j\in{\mathcal H}_\sigma^2; j=1,2,$ denote $F=V_{f_1}^{f_2}.$ If
$\{z\in\mathbb C^n : F(z,k)\neq0\}$ has finite Lebesgue measure for all $k\in K,$
then $F=0.$
\end{proposition}

\begin{proof}
Since $f_j\in{\mathcal H}_\sigma^2,$ we can express
$f_j=\phi_j\otimes h_j,$ where $\phi_j\in L^2(\mathbb{R}^n)$ and
$h_j\in{\mathcal H}_\sigma.$ Then
\begin{eqnarray*}
F(z,k) &=& \left\langle\rho_\sigma(z,k)
f_1,f_2\right\rangle \\
&=& \left\langle\pi(z)\mu(k)\otimes\sigma(k)(\phi_1\otimes h_1),\phi_2\otimes h_2\right\rangle \\
&=& \left\langle\pi(z)\mu(k)\phi_1,\phi_2\right\rangle\left\langle\sigma(k)h_1,h_2\right\rangle\\
&=& \left\langle\pi(z)\psi_1,\phi_2\right\rangle\left\langle\sigma(k)h_1,h_2\right\rangle\\
&=& X(z)\left\langle\sigma(k)h_1,h_2\right\rangle,\\
\end{eqnarray*}
where $\psi_1=\mu(k)\phi_1$ and $X=T(\psi_1,\phi_2).$ If $\left\langle\sigma(k)h_1,h_2\right\rangle=0$
for some $k\in K,$ then $F(.,k)=0.$ On the other hand, if $\left\langle\sigma(k)h_1,h_2\right\rangle\neq0,$
then $X$ is a non-zero function that supported on a set of finite Lebesgue measure. Thus, in view of
Theorem \ref{thm001}, we conclude that $F=0.$
\end{proof}

Next, we prove that if for $F\in L^1\cap L^2(G^\times),$ the operator $W_\sigma(F)$
is of finite rank for each $\sigma\in\hat K,$ then $F=0.$ For proving this, we require
the following crucial results.

\smallskip
For $k\in K,$ define $b_j(k)=\left\langle\sigma(k)\psi_j,\psi_j\right\rangle,$ where $\psi_j\in H_\sigma.$
Then $b_j(e)=\|\psi_j\|^2.$ Set $\|\psi_j\|=\alpha_j.$

\begin{proposition}\label{prop4}
For $\phi_j\in L^2(\mathbb{R}^n);~j\in\{1,\ldots,N\},$ define the function
$\psi$ on $\mathbb C^n$ by $\psi(z,k)=\sum\limits_{j=1}^Nb_j(k)\left\langle\pi(z)\mu(k)\phi_j,\phi_j\right\rangle,$
where $k\in K.$ If $\psi$ is supported on a subset $\mathcal E\times\mathcal F$ of $\mathbb C^n$
such that $0<m(\mathcal E),m(\mathcal F)<\infty,$ then $\psi\equiv0.$
\end{proposition}
\begin{proof}
Let $e$ be the identity element of the group $K.$ Then $\mu(e)=I$ is the identity operator on $L^2(\mathbb{R}^n).$
For $z=x+iy\in\mathbb C^n,$ we write $\psi_y(x)=\psi(z,e).$ Since $\phi_j\in L^2(\mathbb{R}^n),$
there exists a set $A$ of measure zero such that $|\phi_j|$ is finite on $\mathbb{R}^n\smallsetminus A.$
Denote $K_y(\xi)=\sum\limits_{j=1}^N\alpha_j^2\phi_j(\xi+y)\overline{\phi_j(\xi)}$ for almost all $\xi\in\mathbb R^n.$
Then by the hypothesis, $\psi$ can be expressed as
\begin{equation}\label{exp10}
\psi_y(x)=\int_{\mathbb{R}^n}e^{i(x\cdot\xi+\frac{1}{2}x\cdot y)}K_y(\xi)d\xi.
\end{equation}
Since $\psi$ is supported on $\mathcal E\times\mathcal F$ of finite Lebesgue measure,
it follows that $\psi_y=0$ for all $y\in\mathbb R^n\smallsetminus \mathcal F.$ Hence
we infer that $K_y=0,$ whenever $y\in\mathbb R^n\smallsetminus \mathcal F.$

\smallskip

Define the function $\chi$ on $\mathbb R^n\smallsetminus A$ by
$\chi=\left(\alpha_1\phi_1,\ldots,\alpha_N\phi_N\right).$
If $\chi=0$ on  $\mathbb R^n\smallsetminus A,$ then result will follow.
Suppose $\chi\neq0,$ then there exists $\xi_1\in\mathbb R^n\smallsetminus A$ such
that $\chi(\xi_1)\neq0.$

\smallskip

 Now, if it happens that $\chi=0$ on $\mathbb R^n\smallsetminus (B(\xi_1)\cup A),$
where $B(\xi_1)$ is the set $\xi_1+(\mathcal F\cup\{0\}),$ then  $\phi_j$'s are
finitely supported.

\smallskip

Otherwise, we can choose $\xi_l\in\mathbb R^n\smallsetminus\bigcup\limits_{i=1}^{l-1}B(\xi_i)\cup A,$
where $B(\xi_i)=\xi_i+(\mathcal F\cup\{0\})$ such that $\chi(\xi_l)\neq0,$
whenever $l\leq{N}.$ For $l\neq m,$ we have $\xi_l-\xi_m\not\in\mathcal F.$
By the hypothesis,
$K_{\xi_l-\xi_m}(\xi)=\sum\limits_{j=1}^N\alpha_j^2\phi_j(\xi+\xi_l-\xi_m)\overline{\phi_j(\xi)}=0,$
whenever $\xi\in\mathbb R^n\smallsetminus A.$ Hence it follows that $\chi(\xi_l)$ and $\chi(\xi_m)$
are orthogonal. Thus, the set $S=\{\chi(\xi_1),\ldots,\chi(\xi_N)\}$ is an orthogonal set in $\mathbb{C}^{N}.$

\smallskip

Therefore, if $\xi\in\mathbb R^n\smallsetminus\bigcup\limits_{l=1}^{N}\left(B(\xi_l)\cup A\right),$
then $\chi(\xi)\perp S,$ and hence $\chi(\xi)=0.$ Thus, each of $\phi_j$ is supported on a set of
finite Lebesgue measure.

\smallskip

Now, for $k\in K,$ $\psi$ can be expressed as
\[\psi_y(x)=\int_{\mathbb{R}^n}e^{i(x\cdot\xi+\frac{1}{2}x\cdot y)}\left(\sum\limits_{j=1}^Nb_j(k)\chi_j(\xi+y)\overline{\phi_j(\xi)}\right)d\xi,\]
where $\chi_j=\mu(k)\phi_j\in L^2(\mathbb{R}^n).$ Let $H_y(\xi)=\sum\limits_{j=1}^Nb_j(k)\chi_j(\xi+y)\overline{\phi_j(\xi)}.$
Then $H_y$ is finitely supported for all $y\in\mathbb R^n.$ By the Benedicks theorem,
$H_y$ and its Fourier transform both cannot be finitely supported simultaneously.
Hence we conclude that $\psi_y\equiv0$ for all $y\in\mathbb R^n.$
\end{proof}

\begin{remark}
Instead of the rectangle $\mathcal E\times\mathcal F$ in $\mathbb R^{2n}$ if
we consider a set $E$ of finite Lebesgue measure in $\mathbb C^n,$ then
the projection of $E$ on $\mathbb R^n$ need not be a set of finite
measure. Hence the above proof of Proposition \ref{prop4} will not work.
\end{remark}

Let $\mathcal E$ and $\mathcal F$ are Lebesgue measurable subsets of $\mathbb{R}^n$ such that
$0<m(\mathcal E),m(\mathcal F)$ $<\infty$ and $\Sigma=\mathcal E\times\mathcal F.$

\begin{theorem}\label{th1}
Let $F\in L^1\cap L^2(G^\times)$ be supported on $\Sigma\times K.$ If for each $\sigma\in\hat K,$
the operator $W_\sigma(F)$ has finite rank, then $F=0.$
\end{theorem}

\begin{proof}
Let $\bar\tau=F^\ast\times F,$ where $F^\ast(v)=\overline{F\left(v^{-1}\right)}.$
Then $W_\sigma(\bar\tau)=W_\sigma(F)^\ast W_\sigma(F)$ is a positive, finite rank
operator on $\mathcal H_\sigma^2.$ By the spectral theorem, it follows that
\begin{equation}\label{exp01}
W_\sigma(\bar\tau)f=\sum_{j=1}^Na_j\left\langle f,f_j\right\rangle f_j,
\end{equation}
where $\{f_1,\ldots,f_N\}$ is an orthonormal basis for the range of $W_\sigma(\bar\tau)$
which satisfies $W_\sigma(\bar\tau)f_j=a_jf_j$ with $a_j\geq0.$
Now, for $f,g\in \mathcal{H}_\sigma^2,$ we have
\begin{eqnarray}\label{exp02}
\left\langle W_\sigma(\bar\tau)f,g\right\rangle &=& \sum_{j=1}^Na_j\left\langle f,f_j\right\rangle \left\langle f_j,g\right\rangle\\
&=& (2\pi)^{-n}\sum_{j=1}^Na_j\int_{K}\int_{\mathbb{C}^n}V_{f}^{g}(z,k)\overline{V_{f_j}^{f_j}(z,k)}dzdk.\nonumber
\end{eqnarray}
Since $\tau\in L^2(G^\times),$ by Proposition \ref{prop5}, we can write
$\tau=\bigoplus\limits_{\sigma\in\hat K}\tau_\sigma.$ In view of the above
decomposition and by the definition of $W_\sigma(\bar\tau),$ we can write
\begin{eqnarray}\label{exp03}
\left\langle W_\sigma(\bar\tau)f,g\right\rangle &=& \int_{K}\int_{\mathbb{C}^n}\bar\tau(z,k)\left\langle\rho_\sigma^1(z,k)f,g\right\rangle dzdk\\
&=& \int_{K}\int_{\mathbb{C}^n}\overline{{\tau_\sigma}}(z,k)V_{f}^{g}(z,k)dzdk.\nonumber
\end{eqnarray}
Hence, by comparing (\ref{exp02}) with (\ref{exp03}) in view of the orthogonality
relation for the Fourier-Wigner transform as in Lemma \ref{lemma1}, it follows that
\begin{equation}\label{exp04}
{\tau_\sigma}=\sum_{j=1}^N V_{h_j}^{h_j},
\end{equation}
where $h_j=(2\pi)^{-\frac{n}{2}}\sqrt{a_j}f_j\in \mathcal{H}_\sigma^2.$
Now, let $h_j=\phi_j\otimes \psi_j$ for some $\phi_j\in L^2(\mathbb{R}^n)$
and $\psi_j\in{\mathcal H}_\sigma.$ Then from (\ref{exp04}) we have
\begin{eqnarray*}
{\tau_\sigma}(z,k) = \sum_{j=1}^N\left\langle\rho_\sigma(z,k)h_j,h_j\right\rangle
&=& \sum_{j=1}^N\left\langle\pi(z)\mu(k)\phi_j,\phi_j\right\rangle\left\langle\sigma(k)\psi_j,\psi_j\right\rangle\\
&=& \sum_{j=1}^N b_j(k)\left\langle\pi(z)\chi_j,\phi_j\right\rangle,
\end{eqnarray*}
where $\chi_j=\mu(k)\phi_j\in L^2(\mathbb{R}^n)$ and $b_j(k)=\left\langle\sigma(k)\psi_j,\psi_j\right\rangle.$
Since $\bar\tau$ is finitely supported in $\mathbb C^n$ variable, by Proposition
\ref{prop4}, it follows that $\tau_\sigma=0,$ whenever $\sigma\in\hat K.$ In view of
Plancherel formula for the Weyl transform as mentioned in Proposition \ref{prop05},
we conclude that $F=0.$
\end{proof}
Next, we prove Theorem \ref{th2} in the following two cases.

\begin{proof}[Proof of Theorem \ref{th2}]
(i). Since $F\in L^1\cap L^2(G),$ we can write
\begin{eqnarray}\label{exp06}
\hat{F}(\lambda,\sigma) &=& \int_{K}\int_{\mathbb{C}^n}\int_{\mathbb{R}}F(z,t,k)\rho_\sigma(z,t,k)dtdzdk\\
&=& \int_{K}\int_{\mathbb{C}^n}F^\lambda(z,k)\rho_\sigma(z,k)dzdk\nonumber\\
&=& W_\sigma(F^\lambda).\nonumber
\end{eqnarray}

Suppose the operator $W_\sigma(F^\lambda)$ has rank one. Then it is enough
to show that $F^\lambda=0.$ Consider the case when $\lambda=1.$ Since by
hypothesis, $W_\sigma(F^1)$ has rank one, there exist $f_j\in \mathcal{H}_\sigma^2;j=1,2$
such that $W_\sigma(\bar\tau)f=\left\langle f,f_1\right\rangle f_2$ for all
$f\in \mathcal{H}_\sigma^2,$ where $\bar\tau=F^1.$ Hence for
$f,g\in\mathcal{H}_\sigma^2,$
Lemma \ref{lemma1} yields
\begin{eqnarray}\label{exp07}
\left\langle W_\sigma(\bar\tau)f,g\right\rangle &=& \left\langle f,f_1\right\rangle \left\langle f_2,g\right\rangle\\
&=& (2\pi)^{-n}\int_{K}\int_{\mathbb{C}^n}V_{f}^{g}(z,k)\overline{V_{f_1}^{f_2}(z,k)}dzdk\nonumber.
\end{eqnarray}
Let $\tau=\bigoplus\limits_{\sigma\in\hat K}\tau_\sigma,$ where $\tau_\sigma\in V_{B_\sigma}.$
Then by definition of $W_\sigma(\bar\tau),$ it follows that
\begin{eqnarray}\label{exp08}
\left\langle W_\sigma(\bar\tau)f,g\right\rangle &=& \int_{K}\int_{\mathbb{C}^n}\overline{\tau_\sigma}(z,k)V_{f}^{g}(z,k)dzdk.
\end{eqnarray}
Now, by comparing (\ref{exp07}) with (\ref{exp08}) in view of Proposition
\ref{prop5}, we infer that $\tau_\sigma=(2\pi)^{-n}V_{f_1}^{f_2}.$ Finally, by
Proposition \ref{prop1}, it follows that $\tau_\sigma=0$ for all $\sigma\in\hat K.$
That is, $\tau=0$ and hence we conclude that $F=0.$

\smallskip

(ii).
Suppose the operator $W_\sigma(F^\lambda)$ has finite rank. We prove the result for
$\lambda=1$ and the general case will be followed by the scaling argument.
Since  $\hat{F}(1,\sigma)=W_\sigma(F^1),$ by Theorem \ref{th1}, it follows
that $F^1=0.$ Similarly, it can be shown that $F^\lambda=0$ for all $\lambda\in\mathbb R^\ast.$
Thus, we conclude that $F=0.$
\end{proof}

\section{Preliminaries on step two nilpotent group}\label{section5}
In this section, we prove an analogous result of the Benedick's theorem for the
Euclidean Fourier transform on the step two nilpotent Lie groups. However, for
the sake simplicity, we derive the result for the class of groups introduced by
G. M\'{e}tivier (see \cite{M}). These groups are step two nilpotent Lie groups
when quotiented with the hyperplane in the center becomes the Heisenberg group.
The Heisenberg-type groups introduced by  A. Kaplan (see \cite{K}) are examples
of  M\'{e}tivier group.  However, there are  M\'{e}tivier groups which are distinct
from the Heisenberg-type groups. For more details, see \cite{MS}.

\smallskip

Let $G$ be connected, simply connected Lie group with real step two nilpotent
Lie algebra $\mathfrak g.$ Then $\mathfrak g$ has the orthogonal decomposition
$\mathfrak g=\mathfrak b\oplus\mathfrak z,$ where $\mathfrak z$ is the center
of $\mathfrak g.$ Since $\mathfrak g$ is nilpotent, the exponential map
$\exp:\mathfrak g\rightarrow G$ is surjective. Thus, $G$ can be parameterized
by $\mathfrak g=\mathfrak b\oplus\mathfrak z,$ endowed with the exponential
coordinates.

\smallskip

Let $\left\{V_i: i=1,\ldots, m\right\}$ and $\left\{Z_j: j=1,\ldots,k\right\}$
be orthonormal bases of $\mathfrak b$ and $\mathfrak z$ respectively. Then for
$V+Z\in\mathfrak b\oplus\mathfrak z,$ we can identify $g\in G$ with the point
$(V, Z)\in\mathbb R^m\times\mathbb R^k$ such that $g=\exp(V+Z).$ Since
$[\mathfrak b,\mathfrak b]\subset\mathfrak z$ and $[\mathfrak g,[\mathfrak g,\mathfrak g]]=\{0\},$
by the Baker-Campbell-Hausdorff formula, the group law on $G$ can be expressed as
\[\left(V,Z\right)\left(V',Z'\right)=\left(V+V',Z+Z'+\frac{1}{2}[V,V']\right).\]
Let $dV$ and $dZ$ be the Lebesgue measures on $\mathfrak b$ and $\mathfrak z$
respectively. Then the left-invariant Haar measure on $G$ can be expressed as
$dg=dVdZ.$

\smallskip

Now, for $\omega\in\mathfrak z^\ast,$ consider the skew-symmetric bilinear form
$B_\omega$ on $\mathfrak b$ by \[B_\omega(X,Y)=\omega\left([X,Y]\right).\]
Let $m_\omega$ be the orthogonal complement of
\[r_\omega=\left\{X\in\mathfrak b:B_\omega(X,Y)=0,~\forall~ Y\in\mathfrak b\right\}\] in
$\mathfrak b.$ Then $B_\omega$ is called a non-degenerate bilinear form when
$r_\omega$ is trivial. If $B_\omega$ is non-degenerate for all $\omega\neq0,$
then $G$ is called M\'{e}tivier group.

\smallskip

Since $m_\omega$ is invariant under the skew-symmetric bilinear form $B_\omega,$
it follows that the dimension of $m_\omega$ is even.
Let $\Lambda=\{\omega\in\mathfrak z^\ast: \dim m_\omega~\text{is maximum}\}.$
Then $\Lambda$ is a Zariski open subset of $\mathfrak z^\ast$ and for
$\omega\in\Lambda,$ there exists an orthonormal almost symplectic basis
$\left\{X_i(\omega),Y_j(\omega): i=1,\ldots,n\right\}$ of $\mathfrak b$
and $d_i(\omega)>0$ such that
\[\omega[X_i(\omega),Y_j(\omega)]=\left\{
  \begin{array}{ll}
  \delta_{ij}d_i(\omega), & \hbox{ when } X\neq Y;\\
    0, & \hbox{ otherwise}.
  \end{array}
\right.\]
Let $\zeta_\omega=\text{span}\{X_i(\omega): i=1,\ldots,n\}$ and
$\eta_\omega=\text{span}\{Y_j(\omega): j=1,\ldots,n\}.$ Then we can write
$\mathfrak b=\zeta_\omega\oplus\eta_\omega$ and each $(X, Y, Z)\in G$
can be represented by
\[(X, Y, Z)=\sum_{i=1}^nx_i(\omega)X_i(\omega)+\sum_{i=1}^ny_i(\omega)Y_i(\omega)+\sum_{i=1}^kt_i(\omega)Z_i(\omega).\]
Hence a typical element of $G$ can be written as $(x,y,t),$ where $x,y\in\mathbb R^n$
and $t\in\mathbb R^k.$ For more details, we refer to \cite{CG,M,Ki}.

\smallskip

Next, we briefly describe the irreducible representation of the M\'{e}tivier
group $G$ which can be parameterized by $\Lambda.$ That is, each $\omega\in\Lambda$
induces an irreducible unitary representation $\pi_{\omega}$ of $G$ by
\[\left(\pi_{\omega}(x,y,t)\phi\right)(\xi)=e^{i\sum_{j=1}^k\omega_jt_j
+i\sum_{j=1}^n d_j(\omega)\left(x_j\xi_j+\frac{1}{2}x_jy_j\right)}\phi(\xi+y),\]
whenever $\phi\in L^2(\eta_\omega).$ For the sake of simplicity, we write
$v=(x,y).$ Then the group Fourier transform of $f\in L^1(G)$ can be defined by
\[\hat f(\omega)=\int_{\mathfrak z}\int_{\mathfrak b}f(v,t)\pi_{\omega}(v,t)dvdt,\]
where $\omega\in\Lambda.$ Now, we define the Fourier inversion of $f$ in the
$t$ variable by
\[f^\omega(v)=\int_{\mathfrak z}e^{i\sum_{j=1}^k\omega_jt_j}f(v,t)dt.\]
Then for the suitable functions $f$ and $g$ on $\mathfrak b,$ we can
define the $\omega$-twisted convolution of $f$ and $g$ by
\[f\ast_\omega g(v)=\int_b f(v-v')g(v')e^{\frac{i}{2}\omega([v,v'])}dv'.\]
Here it is immediate that $(f\ast g)^\omega=f^\omega\ast_\omega g^\omega.$
Let $p(\omega)=\Pi_{i=1}^nd_i(\omega)$ be the symmetric
function of degree $n$ corresponding to $B_\omega.$ For $f\in L^1\cap L^2(G),$
the operator $\hat f(\omega)$ is a Hilbert-Schmidt operator that satisfies
\[p(\omega)\|\hat f(\omega)\|_{HS}^2=(2\pi)^{n}\int_{\mathfrak b}|f^\omega(v)|^2dv.\]
Denote $\pi_{\omega}(v)=\pi_{\omega}(v,o).$ Then the Fourier inversion
$f^\omega$ can be determined by the formula
\[f^\omega(v)=(2\pi)^{-n}p(\omega)~\text{tr}(\pi_{\omega}(v)^\ast\hat f(\omega)).\]

\section{Uniqueness results on step two nilpotent group}\label{section6}
For $\omega\in\Lambda$ and $h\in L^1\cap L^2(\mathfrak b),$ the Weyl transform
$W_{\omega}(h)$ is defined by
\begin{equation}\label{exp22}
W_{\omega}(h)=\int_{\mathfrak b}h(v)\pi_{\omega}(v)dv.
\end{equation}
The Weyl transform $W_{\omega}(h)$ is a Hilbert-Schmidt operator on $L^2(\eta_\omega)$
that satisfies the following Plancherel formula, (see \cite{PT}).

\begin{theorem}\label{th3}
For $h\in L^2(\mathfrak b),$ the following equality holds:
\[p(\omega)\|W_{\omega}(h)\|_{HS}^2=(2\pi)^{n}\int_{\mathfrak b}|h(v)|^2dv.\]
\end{theorem}

\begin{proposition}\label{prop22}
For $h\in L^1\cap L^2(\mathfrak b),$ we have the identities:\\
$(i)~W_{\omega}(h^\ast)=W_{\omega}(h)^\ast,$ where $h^\ast(v)=\overline{h(v^{-1})},$\\
$(ii)~W_{\omega}(h^\ast\ast_\omega h)=W_{\omega}(h)^\ast W_{\omega}(h).$
\end{proposition}

\begin{proof}
$(i)$ For $\phi,\psi\in L^2(\eta_\omega),$ we can write
\begin{eqnarray*}
\left\langle W_{\omega}(h^\ast)\phi,\psi\right\rangle &=&\int_{\mathfrak b}
h^\ast(v)\left\langle\pi_{\omega}(v)\phi,\psi\right\rangle dv\\
&=&\int_{\mathfrak b}\left\langle \phi,h\left(v^{-1}\right)\pi_{\omega}\left(v^{-1}\right)g\right\rangle dv\\
&=& \left\langle \phi,W_{\omega}(h)\psi\right\rangle=\left\langle W_{\omega}(h)^\ast\phi,\psi\right\rangle.
\end{eqnarray*}

$(ii)$ Further, we have
\begin{eqnarray*}
\left\langle W_{\omega}(h)^\ast W_{\omega}(h)\phi,\psi\right\rangle &=&\int_{\mathfrak b}\int_{\mathfrak b}
h^\ast(v)h(v')\left\langle\pi_{\omega}(v)\pi_{\omega}(v')\phi,\psi\right\rangle dvdv'\\
&=&\int_{\mathfrak b}\int_{\mathfrak b}h^\ast(v-v')h(v')e^{\frac{i}{2}\omega([v,v'])}\left\langle\pi_{\omega}(v)\phi,\psi\right\rangle dvdv'\\
&=&\int_{\mathfrak b}(h^\ast\ast_\omega h)(v)\left\langle\pi_{\omega}(v)\phi,\psi\right\rangle dv\\
&=&\left\langle W_{\omega}(h^\ast\ast_\omega h)\phi,\psi\right\rangle.
\end{eqnarray*}
\end{proof}

Now, we state our main result of this section. Let $\text{A}$ and $\text{B}$ are
Lebesgue measurable subsets of $\zeta_\omega$ and $\eta_\omega$ respectively such
that $0<m(\text{A}),m(\text{B})<\infty,$ where $m$ denotes the Lebesgue  measure.

\begin{theorem}\label{th4}
Suppose $f\in L^1(G)$ is supported on the set $\Sigma\oplus\mathfrak z,$ where
$\Sigma$ is a subset of $\mathfrak b.$
\smallskip

$(i)$ If $~\Sigma$ has finite Lebesgue measure and $\hat{f}(\omega)$
is a rank one operator for all $\omega\in\Lambda,$ then $f=0.$
\smallskip

$(ii)$ If $~\Sigma=A\times B$ and $\hat{f}(\omega)$ has finite rank
for all $\omega\in\Lambda,$ then $f=0.$
\end{theorem}

In order to prove Theorem \ref{th4}, we need the following crucial results.
Let $\phi,\psi\in L^2(\eta_\omega).$ Then the Fourier-Wigner transform of $\phi$
and $\psi$ is a function on $\mathfrak b$ defined by
\[T(\phi,\psi)(v)=\left\langle\pi_{\omega}(v)\phi,\psi\right\rangle.\]
As a consequence of the Schur's orthogonality relation, these functions
$T(\phi,\psi)$'s are orthogonal among themselves. For more details, we refer
to Wolf \cite{W}.

\begin{lemma}\label{lemma3}{\em\cite{W}}
Let $\phi_j,\psi_j\in L^2(\eta_\omega);~j=1,2.$ Then
\[\int_{\mathfrak b}T(\phi_1,\psi_1)(v)\overline{T(\phi_2,\psi_2)(v)}dv=
c(\omega)\left\langle \phi_1,\phi_2 \right\rangle\overline{\left\langle \psi_1,\psi_2 \right\rangle},\]
where $c(\omega)=(2\pi)^{n}~p(\omega)^{-1}.$
\end{lemma}

We observe that these functions $T(\phi,\psi)$'s generate an orthonormal basis for
$L^2(\mathfrak b).$ Let $\{\varphi_j:j\in\mathbb N\}$ be an orthonormal basis for
$L^2(\eta_\omega).$
\begin{proposition}\label{prop6}
The set $\{T(\varphi_i,\varphi_j):i,j\in\mathbb N\}$ is an orthonormal basis
for $L^2(\mathfrak b).$
\end{proposition}

\begin{proof}
In view of Lemma \ref{lemma3}, it is clear that $\{T(\varphi_i,\varphi_j):i,j\in\mathbb N\}$
is an orthonormal set. Now, it only remains to verify the completeness. For this, let
$f\in L^2(\mathfrak b)$ be such that $\left\langle f,T(\varphi_i,\varphi_j)\right\rangle=0,$
whenever $i,j\in\mathbb N.$ Then
\begin{eqnarray}\label{eq012}
\left\langle W_{\omega}(\bar f)\phi_i,\phi_j\right\rangle &=& \int_{\mathfrak b}
\bar f(v)\left\langle\pi_{\omega}(v)\phi_i,\phi_j\right\rangle dv\\
&=& \left\langle f,T(\phi_i,\phi_j)\right\rangle=0.\nonumber
\end{eqnarray}
Hence, we infer that $W_{\omega}(\bar f)=0.$ Thus, by the Plancherel
Theorem \ref{th3}, we conclude that $f=0.$
\end{proof}

\begin{proposition}\label{prop21}
Let $F=T(\phi,\psi),$ where $\phi,\psi\in L^2(\eta_\omega).$ If the set
$\{v\in\mathfrak b: F(v)\neq0\}$ has a finite Lebesgue measure, then $F$
has to vanish identically.
\end{proposition}
\begin{proof}
We would like to mention that the proof of Proposition \ref{prop21} is almost
similar to Theorem \ref{thm001} and hence we omit it here.
\end{proof}

Let $\mathcal E$ and $\mathcal F$ be finite measure subset of $\zeta_\omega$ and $\eta_\omega$
respectively such that $0<m(\mathcal E),m(\mathcal F)<\infty$ and $\Sigma=\mathcal E\times\mathcal F.$
\begin{lemma}\label{lemma21}
For $h_j\in L^2(\eta_\omega),$ write $K_y(\xi)=\sum\limits_{j=1}^N h_j(\xi+y)\overline{h_j(\xi)},$
where $y\in\eta_\omega.$ If $K_y(\xi)=0$ for all $y\in\eta_\omega\smallsetminus\mathcal F$ and for
almost all $\xi\in\eta_\omega,$ then each of $h_j$ is finitely supported.
\end{lemma}
\begin{proof}
Since $h_j\in L^2(\eta_\omega),$  there exists a set $A$ of Lebesgue measure
zero such that $|h_j|$ is finite on $\eta_\omega\smallsetminus A.$ Define a
function $\chi$ on $\eta_\omega\smallsetminus A$ by\[\chi=\left(h_1,\ldots,h_N\right).\]
If $h_j$ is non-vanishing on $\eta_\omega\smallsetminus A$ for some $j,$ then
we can choose $\xi_1\in\eta_\omega\smallsetminus A$ such that $\chi(\xi_1)\neq0.$
Let $B(\xi_1)$ be the set $\xi_1+(\mathcal F\cup\{0\}).$
If $\chi$ vanishes on $\eta_\omega\smallsetminus B(\xi_1)\cup A,$
then the result follows. Otherwise, by induction, we can choose
$\xi_j\in\eta_\omega\smallsetminus\bigcup\limits_{i=1}^{j-1}\left(B(\xi_i)\cup A\right)$
such that $\chi(\xi_j)\neq0,$ whenever $j\leq N,$ where $B(\xi_i)=\xi_i+(\mathcal F\cup\{0\}).$
Thus by the hypothesis, the set $S=\{\chi(\xi_j): j=1,2,\ldots,N\}$ is an orthogonal set in $\mathbb C^N.$
Now, if $\xi\in\eta_\omega\smallsetminus\bigcup\limits_{j=1}^{N}\left(B(\xi_j)\cup A\right),$
then $\chi(\xi)\in S^{\perp},$ and hence $\chi(\xi)=0.$
\end{proof}

\begin{proposition}\label{prop3}
Let $h\in L^1\cap L^2(\mathfrak b)$ be supported on $\Sigma$ in $\mathfrak b.$
If $~W_{\omega}(h)$ is a finite rank operator, then $h=0.$
\end{proposition}

\begin{proof}
Let $\bar\tau=h^\ast\ast_\omega h,$ where $h^\ast(v)=\overline{h\left(v^{-1}\right)}.$
Then $W_\omega(\bar\tau)=W_\omega(h)^\ast W_\omega(h)$ is a positive and finite rank
operator on $L^2(\eta_\omega).$ By the spectral theorem,  there exist an orthonormal
set $\{\phi_j\in L^2(\eta_\omega): j=1,\ldots,N\}$ and scalars $a_j\geq0$ such that
\[W_{\omega}(\bar\tau)\phi=\sum_{j=1}^Na_j\left\langle \phi,\phi_j\right\rangle \phi_j,\]
whenever $\phi\in L^2(\eta_\omega).$ Now, for $\psi\in L^2(\eta_\omega),$ we have
\begin{eqnarray}\label{exp24}
\left\langle W_{\omega}(\bar\tau)\phi,\psi\right\rangle &=&\sum_{j=1}^Na_j\left\langle\phi,\phi_j\right\rangle\left\langle\phi_j,\psi\right\rangle\nonumber\\
&=& c(\omega)^{-1}\sum_{j=1}^Na_j\int_{\mathfrak b}T(\phi,\psi)(v)\overline{T(\phi_j,\phi_j)(v)}dv.
\end{eqnarray}
Further, by definition of $W_{\omega}(\bar\tau),$ we have
\begin{equation}\label{exp25}
\left\langle W_{\omega}(\bar\tau)\phi,\psi\right\rangle=\int_{\mathfrak b}\bar\tau(v)T(\phi,\psi)(v)dv.
\end{equation}
Hence, by comparing (\ref{exp24}) with (\ref{exp25}) in view of Proposition \ref{prop6},
it follows that
\begin{equation}\label{exp35}
\tau=\sum\limits_{j=1}^N T(h_j,h_j),
\end{equation}
where $h_j=c(\omega)^{-\frac{1}{2}}\sqrt{a_j}~\phi_j\in L^2(\eta_\omega).$  Now, for
$v=(x,y),$ write $\tau_y(x)=\tau(x,y).$ Then Equation (\ref{exp35})
becomes
\begin{equation}\label{exp27}
\tau_y(x)=\int_{\eta_\omega}e^{i\sum\limits_{j=1}^nd_j(\omega)
(x_j\xi_j+\frac{1}{2}x_jy_j)}K_y(\xi)d\xi.
\end{equation}
Since $\bar\tau$ is supported on $\mathcal E\times\mathcal F$, it follows that
$K_y(\xi)=0$ for almost every $\xi$ and for all $y\in\eta_\omega\smallsetminus\mathcal F.$
Then in view of Lemma \ref{lemma21}, it follows that each of $h_j$ is finitely supported
and hence each of $K_y$ is finitely supported. Since $\tau_y$ is is supported on $\mathcal E,$
whenever $y\in\eta_\omega,$ we infer that $\tau_y$ is zero for all
$y\in\eta_\omega.$ Now, by Plancherel Theorem \ref{th3}, we conclude that $h=0.$
\end{proof}

\begin{proof}[Proof of Theorem \ref{th4}]
(i). By a simple calculation, we get
\[\hat{f}(\omega)=\int_{\mathfrak b}f^\omega(v)\pi_{\omega}(v)dv=W_{\omega}(f^\omega).\]
Since $f^\omega$ is finitely supported  and the operator $W_{\omega}(f^\omega)$ has finite
rank, by Proposition \ref{prop3}, it follows that $f^\omega=0,$ whenever $\omega\in\Lambda.$
Hence we infer $f=0.$
\smallskip

(ii). It is enough to prove that if $W_{\omega}(f^\omega)$ has rank one, then $f^\omega=0.$ Let
$W_{\omega}(f^\omega)$ be a rank one operator. Then there exist $\phi_j\in L^2(\eta_\omega);~j=1,2$
such that \[W_{\omega}(\bar\tau)\phi=\left\langle \phi,\phi_1\right\rangle \phi_2\] for all
$\phi\in L^2(\eta_\omega),$ where $\bar\tau=f^\omega.$ Thus, for $\psi\in L^2(\eta_\omega),$
it follows that
\begin{eqnarray}\label{exp29}
\left\langle W_{\omega}(\bar\tau)\phi,\psi\right\rangle &=& \left\langle \phi,\phi_1\right\rangle \left\langle \phi_2,\psi\right\rangle\nonumber\\
&=& c(\omega)^{-1}\int_{\eta_\omega}\int_{\zeta_\omega}
T(\phi,\psi)(v)\overline{T(\phi_1,\phi_2)(v)}dv.
\end{eqnarray}
Further,  by definition, we get
\begin{eqnarray}\label{exp30}
\left\langle W_{\omega}(\bar\tau)\phi,\psi\right\rangle &=& \int_{\zeta_\omega}\int_{\eta_\omega}
\bar\tau(v)T(\phi,\psi)(v)dv.
\end{eqnarray}
Hence by comparing (\ref{exp29}) with (\ref{exp30}) in view of Lemma \ref{lemma3},
we infer that \[\tau(v)=c(\omega)^{-1}T(\phi_1,\phi_2)(v).\] Thus, from Proposition
\ref{prop21}, it follows that $\tau\equiv0.$
\end{proof}

\smallskip

\noindent{\bf Concluding remarks:}\\

\noindent  If the Fourier transform of a compactly supported
function $f$ on $\mathbb H^n\ltimes U(n)$ (or step two nilpotent
Lie groups) lands into the space of compact operators, then $f$
might be zero. However, it would be a good question to consider
the case when the spectrum of the Fourier transform of a compactly
supported function is supported on a thin uncountable set.

\bigskip

\noindent{\bf Acknowledgements:}\\
The authors would like to gratefully acknowledge the support provided by
IIT Guwahati, Government of India.

\bigskip


\end{document}